\documentclass{amsart}
\usepackage{amsmath} 
\usepackage{amsfonts} 
\usepackage{amssymb} 
\usepackage[utf8]{inputenc} 
\usepackage[british]{babel} 
\usepackage{hyperref} 
\usepackage[protrusion=true,expansion=false]{microtype} 
\usepackage{enumitem} 
\usepackage[T1]{fontenc}
\usepackage{amsthm} 
\usepackage{dsfont}
\usepackage{lmodern}
\usepackage{xcolor}
\usepackage{mathrsfs}

\definecolor{dblue}{rgb}{0,0,0.70}
\definecolor{dgreen}{rgb}{0,0.60,0}
\hypersetup{
	unicode=true,
	colorlinks=true,
	citecolor=dgreen,
	linkcolor=dblue,
	anchorcolor=dblue
}

\newcommand{\1}{\ensuremath{\mathds{1}}}
\newcommand{\comp}{\ensuremath{\mathrel{\|}}}
\newcommand{\forces}{\ensuremath{\mathrel{\Vdash}}}

\newcommand{\dda}{\ensuremath{\dot{a}}}
\newcommand{\ddx}{\ensuremath{\dot{x}}}
\newcommand{\ddy}{\ensuremath{\dot{y}}}
\newcommand{\ddA}{\ensuremath{\dot{A}}}

\newcommand{\PGF}{\ensuremath{\left\langle\mathbb{P},\mathscr{G},\mathscr{F}\right\rangle}}

\newcommand{\abs}[1]{\ensuremath{\left|#1\right|}}
\newcommand{\cabs}[1]{\ensuremath{\left\lceil#1\right\rceil}}
\newcommand{\parenth}[1]{\ensuremath{\left(#1\right)}}
\newcommand{\tup}[1]{\ensuremath{\langle#1\rangle}}

\newcommand*{\defeq}{\mathrel{\vcenter{\baselineskip0.5ex \lineskiplimit0pt
                     \hbox{\scriptsize.}\hbox{\scriptsize.}}}%
                     =}
\newcommand{\lomega}{\ensuremath{{{<}\omega}}}
\newcommand{\midcup}{\ensuremath{{\textstyle\bigcup}}}
\newcommand{\res}{\ensuremath{\nobreak\mskip2mu\mathpunct{}\nonscript
  \mkern-\thinmuskip{\upharpoonright}\mskip6muplus1mu\relax}} 
\newcommand{\vphi}{\ensuremath{\varphi}}
\newcommand{\Inj}[2]{\ensuremath{\operatorname{Inj}(#1,#2)}}

\renewcommand{\leq}{\leqslant}
\renewcommand{\nleq}{\nleqslant}
\renewcommand{\geq}{\geqslant}

\DeclareMathOperator{\Add}{Add}
\DeclareMathOperator{\Aut}{Aut}
\DeclareMathOperator{\Card}{Card}
\DeclareMathOperator{\cf}{cf}
\DeclareMathOperator{\dom}{dom}
\DeclareMathOperator{\fix}{fix}
\DeclareMathOperator{\id}{id}
\DeclareMathOperator{\Ord}{Ord}
\DeclareMathOperator{\Spec}{Spec}
\DeclareMathOperator{\Suc}{SC}
\DeclareMathOperator{\sym}{sym}
\DeclareMathOperator{\first}{1^{\text{\textup{st}}}}

\newcommand{\AC}{\ensuremath{\mathsf{AC}}}
\newcommand{\DC}{\ensuremath{\mathsf{DC}}}
\newcommand{\HS}{\ensuremath{\mathsf{HS}}}
\newcommand{\SVC}{\ensuremath{\mathsf{SVC}}}
\newcommand{\WO}{\ensuremath{\mathsf{WO}}}
\newcommand{\ZF}{\ensuremath{\mathsf{ZF}}}
\newcommand{\ZFC}{\ensuremath{\mathsf{ZFC}}}

\newcommand{\bbP}{\ensuremath{\mathbb{P}}}

\newcommand{\sF}{\ensuremath{\mathscr{F}}}
\newcommand{\sG}{\ensuremath{\mathscr{G}}}
\newcommand{\power}{\ensuremath{\mathscr{P}}}

\newcommand{\frakC}{\ensuremath{\mathfrak{C}}}
\newcommand{\frakD}{\ensuremath{\mathfrak{D}}}
\newcommand{\frakU}{\ensuremath{\mathfrak{U}}}

\newtheoremstyle{boldrk}
  {.5\baselineskip plus .2\baselineskip minus .2\baselineskip}%
  {.5\baselineskip plus .2\baselineskip minus .2\baselineskip}%
  {}{}{\bfseries}{.}%
  {5pt plus 1pt minus 1pt}{}%

\theoremstyle{plain}
\newtheorem{thm}{Theorem}[section]
\newtheorem*{mainthm}{Main Theorem}
\newtheorem{lem}[thm]{Lemma}
\newtheorem{prop}[thm]{Proposition}
\newtheorem{cor}[thm]{Corollary}
\newtheorem{fact}[thm]{Fact}

\theoremstyle{boldrk}
\newtheorem*{rk}{Remark}
\newtheorem*{eg}{Example}
\newtheorem{qn}[thm]{Question}

\theoremstyle{definition}
\newtheorem{defn}[thm]{Definition}

\title[The {H}artogs--{L}indenbaum Spectrum]{The {H}artogs--{L}indenbaum Spectrum\linebreak{}of Symmetric Extensions}
\author{Calliope Ryan-Smith}

\email{c.Ryan-Smith@leeds.ac.uk}
\urladdr{https://academic.calliope.mx}

\address{School of Mathematics, University of Leeds, LS2 9JT, UK}

\date{15th August 2024}
\thanks{The author's work was financially supported by EPSRC via the Mathematical Sciences Doctoral Training Partnership, grant number EP/W523860/1. For the purposes of open access, the author has applied a Creative Commons Attribution (CC BY) licence to any Author Accepted Manuscript version arising from this submission. No data are associated with this article.}

\begin{document}

\begin{abstract}
We expand the classic result that $\AC_\WO$ is equivalent to the statement ``For all $X$, $\aleph(X)=\aleph^*(X)$'' by proving the equivalence of many more related statements. Then, we introduce the Hartogs--Lindenbaum spectrum of a model of $\ZF$, and inspect the structure of these spectra in models that are obtained by a symmetric extension of a model of $\ZFC$. We prove that all such spectra fall into a very rigid pattern.
\end{abstract}

\maketitle

\section{Introduction}

Perhaps one of the most powerful consequences of the axiom of choice is a straightforward classification of size for all sets: Given any set $X$, there is a minimum ordinal $\alpha$ such that $X$ has the same cardinality as $\alpha$. Indeed, grappling with how to compare sizes of infinite objects (or appropriate abstractions of this notion) is an inalienable aspect of modern mathematical foundations. Even without the axiom of choice, it is still understood that if $Y$ is a superset of $X$, then $Y$ is still `at least as big as $X$' in some way. Taking this further, we are able to compare the cardinalities of sets through functions that map between them. If there is a surjection from $X$ onto $\omega$, then we can still describe $X$ as being `at least as big as $\omega$'. Even when we lose the straightforward description of cardinality classes that is obtained from the axiom of choice, we can still consider comparisons between sets and ordinals in this way.
\begin{defn}
Let $X$ be a set. The \emph{Hartogs number} of $X$ is
\begin{equation*}
\aleph(X)\defeq\min\{\alpha\in\Ord\mid\text{ There is no injection }f\colon\alpha\to X\}.
\end{equation*}
The \emph{Lindenbaum number} of $X$ is
\begin{equation*}
\aleph^*(X)\defeq\min\{\alpha\in\Ord\backslash\{0\}\mid\text{ There is no surjection }f\colon X\to\alpha\}.
\end{equation*}
\end{defn}
\noindent If $X$ is well-orderable, then $\aleph(X)=\aleph^*(X)=\abs{X}^+$, and so in the case that $\AC$ holds, these descriptions of sets tell us no more than cardinality already did. However, if the axiom of choice does not hold and $X$ is not well-orderable then such descriptions still provide insight into the cardinality of $X$. Indeed the existence of $\aleph(X)$ and $\aleph^*(X)$ is a theorem of $\ZF$, in the former case by Hartogs's lemma and in the latter by a lemma of Lindenbaum's theorem.\footnote{Hartogs's lemma is proved in \cite{hartogs_uber_1915}. Lindenbaum's theorem is stated in \cite[Th\'eor\`eme~82.$A_6$]{lindenbaum_communication_1926} with the first published proof in \cite{sierpinski_sur-lindenbaum_1947}.} It is easy to see, though important to note, that they must both be cardinal numbers and that for all $X$, $\aleph(X)\leq\aleph^*(X)$.

As observed, if we assume $\AC$ then $\aleph(X)=\aleph^*(X)$ for all sets $X$, but this principle is in general weaker than the axiom of choice. In fact, $(\forall X)\aleph(X)=\aleph^*(X)$ is equivalent to the axiom of choice for well-ordered families of non-empty sets. This equivalence, a theorem of Pincus, uses a powerful construction that takes sets $X$ with $\aleph(X)\neq\aleph^*(X)$ and `transfers' this property to sets of larger Hartogs or Lindenbaum number. We refer to sets $X$ such that $\aleph(X)\neq\aleph^\ast(X)$ as \emph{eccentric}. In this paper, we shall fine-tune this construction to produce many more equivalent statements.

\begin{thm}\label{thm:acwo-equiv-pelc-intro}
The following are equivalent:
\begin{enumerate}[label=\textup{(\arabic*)}]
\item For all $X$, $\aleph(X)=\aleph^*(X)$;
\item there is $\kappa$ such that for all $X$, $\aleph^*(X)\geq\kappa\implies\aleph(X)=\aleph^*(X)$;
\item there is $\kappa$ such that for all $X$, $\aleph(X)\geq\kappa\implies\aleph(X)=\aleph^*(X)$;
\item $\AC_\WO$;
\item for all $X$, $\aleph(X)$ is a successor; and
\item for all $X$, $\aleph(X)$ is regular.
\end{enumerate}
\end{thm}

\begin{rk}
While we would like to include the statements ``for all $X$, $\aleph^*(X)$ is a successor'' and ``for all $X$, $\aleph^*(X)$ is regular'' in our theorem,\footnote{For symmetry, if nothing else.} this does not hold: In Cohen's first model, $\AC_\WO$ fails, but $\aleph^*(X)$ is a regular successor for all $X$. This is proved in Section~\ref{s:spetra;ss:cohens} as Corollary~\ref{cor:no-reverse-implication}.
\end{rk}

To take Theorem~\ref{thm:acwo-equiv-pelc-intro} further, given that $\AC_\WO$ is not a consequence of $\ZF$, it is quite possible to build models of $\ZF$ in which we have sets $X$ such that ${\aleph(X)\neq\aleph^*(X)}$. Let us produce a classification tool for such objects.

\begin{defn}[Hartogs--Lindenbaum Spectrum]
Given a model $M$ of $\ZF$, the \emph{Hartogs--Lindenbaum spectrum} (or simply \emph{spectrum}) of $M$, denoted by $\Spec_\aleph(M)$, is the class
\begin{equation*}
\Spec_\aleph(M)\defeq\{\tup{\lambda,\kappa}\mid(\exists X)\aleph(X)=\lambda,\aleph^*(X)=\kappa\}.
\end{equation*}
\end{defn}

This paper explores the possible spectra of models of $\ZF$ that arise as symmetric extensions of models of $\AC$. This behaviour is captured internally to a model by \emph{small violations of choice}, or $\SVC$. This axiom, introduced in \cite{blass_injectivity_1979}, is the statement ``There is a set $S$ such that for all $X$ there is an ordinal $\eta$ and a surjection ${f\colon S\times\eta\to X}$'', but is also surprisingly a precise description of being a symmetric extension of a model of $\AC$, by \cite{usuba_choiceless_2021}. Indeed, $\SVC$ is equivalent to several statements relating to symmetric extensions and $\AC$, and is introduced more thoroughly in Section~\ref{s:preliminaries;ss:svc}.

In any model of $\SVC$, the Hartogs--Lindenbaum spectrum is broken down into four parts.

\begin{mainthm}
Let $M\vDash\SVC$. Then there are cardinals $\phi\leq\psi\leq\chi_0\leq\Omega$, a cardinal $\psi^*\geq\psi$, a cardinal $\chi\in[\chi_0,\chi_0^+]$, and a set $C\subseteq[\phi,\chi_0)$ such that
\begin{equation*}
\Spec_\aleph(M)=\bigcup\left\{\begin{alignedat}{2}
\Suc&=\{\tup{\lambda^+,\lambda^+}&&\mid\lambda\in\Card\}\\
\frakD&\subseteq\{\tup{\lambda,\kappa}&&\mid\psi\leq\lambda\leq\kappa\leq\chi,\psi^*\leq\kappa\}\\
\frakC&\subseteq\{\tup{\lambda,\lambda^+}&&\mid\cf(\lambda)\in C,\lambda<\Omega\}\\
\frakU&=\{\tup{\lambda,\lambda^+}&&\mid\cf(\lambda)\in C,\lambda\geq\Omega\}.
\end{alignedat}\right.
\end{equation*}
\end{mainthm}

That is, there is a necessary core to the spectrum $\{\tup{\lambda^+,\lambda^+}\mid\lambda\in\Card\}$ that every model of $\ZF$ contains, since $\aleph(\lambda)=\aleph^*(\lambda)=\lambda^+$ for all cardinals $\lambda$; there is a bounded chaotic part of the spectrum containing those $\tup{\lambda,\kappa}$ that have no restrictions other than $\psi\leq\lambda\leq\kappa\leq\chi$, and $\psi^*\leq\kappa$; there is a bounded but potentially irregular part of the spectrum in which the only values are $\tup{\lambda,\lambda^+}$ for those $\lambda$ such that $\cf(\lambda)\in C$, with $\lambda<\Omega$; and there is an unbounded, controlled tail of the spectrum containing precisely $\tup{\lambda,\lambda^+}$ for all $\lambda$ such that $\cf(\lambda)\in C$.

\subsection{Structure of the paper}

Section~\ref{s:preliminaries} establishes preliminaries for the paper. No knowledge of forcing will be needed for the results in Section~\ref{s:ac-wo} or Section~\ref{s:spectra;ss:svc}; the results can be understood entirely using the framework of transitive nested models of $\ZF$. However, Section~\ref{s:spetra;ss:cohens} uses the forcing framework, so we briefly introduce our standard treatment of forcing in Section~\ref{s:preliminaries;ss:forcing}. In particular, we introduce all necessary concepts for our use of $\SVC$.

In Section~\ref{s:ac-wo} we investigate the equivalence between $\AC_\WO$, $\Spec_\aleph=\Suc$, and several other equivalent conditions using core machinery that allows an `upwards transfer' of eccentricity. In Section~\ref{s:spectra} we investigate the spectrum of models of $\SVC$, providing bounds for the behaviour of such models and some instances of sharpness on those bounds.

\section{Preliminaries}\label{s:preliminaries}

\noindent Throughout this paper we work in $\ZF$. Given a set $X$, we denote by $\abs{X}$ its cardinal number. If $X$ can be well-ordered, then $\abs{X}$ is simply the least ordinal $\alpha$ such that a bijection between $\alpha$ and $X$ exists. Otherwise, we use the Scott cardinal of $X$, the set $\{Y\in V_\alpha\mid\exists f\colon X\to Y\text{ a bijection}\}$ with $\alpha$ taken minimal such that the set is non-empty. Greek letters, when used as cardinals, always refer to well-ordered cardinals. We call an ordinal $\alpha$ a cardinal if $\abs{\alpha}=\alpha$, and we shall denote by $\Card$ the class of all well-ordered cardinals.

We write $\abs{X}\leq\abs{Y}$ to mean that there is an injection from $Y$ to $X$, and ${\abs{X}\leq^*\abs{Y}}$ to mean that there is a surjection from $Y$ to $X$ or that $X$ is empty. These notations extend to $\abs{X}<\abs{Y}$ (and $\abs{X}<^*\abs{Y}$) to mean that $\abs{X}\leq\abs{Y}$ (respectively $\abs{X}\leq^*\abs{Y}$) and there is no injection from $Y$ to $X$ (respectively surjection from $X$ to $Y$). Finally, $\abs{X}=\abs{Y}$ means that there is a bijection between $X$ and $Y$.

Using this notation, one may redefine the Hartogs and Lindenbaum numbers as
\begin{align*}
\aleph(X)&\defeq\min\{\alpha\in\Ord\mid\abs{\alpha}\nleq\abs{X}\}\text{ and}\\
\aleph^*(X)&\defeq\min\{\alpha\in\Ord\mid\abs{\alpha}\nleq^*\abs{X}\}.
\end{align*}
Given two cardinals $\lambda,\kappa$, we will denote by $[\lambda,\kappa]$ the set $\{\mu\in\Card\mid\lambda\leq\mu\leq\kappa\}$. Similarly, $(\lambda,\kappa)$ will be the open interval, and $[\lambda,\kappa)$ and $(\lambda,\kappa]$ are the one-sided intervals. If $\alpha,\beta$ are instead understood to be ordinals, then we will take the intervals over the ordinals, so $[\alpha,\beta]=\{\gamma\in\Ord\mid\alpha\leq\gamma\leq\beta\}$, and so forth. Note that if $\lambda=\kappa$, then $[\lambda,\kappa)=\emptyset$.

Given two sets $A$, $B$, we shall denote by $\Inj{A}{B}$ the set of injections $A\to B$.

\subsection{Choice-like axioms}

Throughout, we shall use several axioms that can be considered to be partial fulfilments of the full strength of $\AC$. Recall that if $X$ is a set of non-empty sets, a \emph{choice function} for $X$ is a function $c\colon X\to\bigcup X$ such that for all $x\in X$, $f(x)\in x$.

\subsubsection{The axiom of choice $\AC_X$}
For any set $X$, we shall denote by $\AC_X$ the statement that all families of non-empty sets indexed by $X$ admit a choice function. If $\alpha$ is an ordinal or a cardinal, we shall denote by $\AC_{{<}\alpha}$ the statement that all families of non-empty sets indexed by some $\beta<\alpha$ admit a choice function. Finally, by $\AC_\WO$, we mean the statement that all well-orderable families of non-empty sets admit a choice function; equivalently, this can be written $(\forall\lambda\in\Card)\AC_\lambda$. When the subscript is omitted, we mean the full axiom of choice: every family of non-empty sets admits a choice function.

\subsubsection{Dependent choice}
We say that a partially ordered set $(T,\leq)$ is a \emph{tree} if $T$ has a minimum element and, for all $t\in T$, the set $\{s\in T\mid s\leq t\}$ is well-ordered by $\leq$. Given a cardinal $\lambda$, we say that $T$ is \emph{$\lambda$-closed} if all $\leq$-chains in $T$ of length less than $\lambda$ have an upper bound in $T$. Finally, the statement $\DC_\lambda$, the \emph{principle of dependent choice} (for $\lambda$), is the statement that every $\lambda$-closed tree has a maximal element or a chain of order type $\lambda$. When the subscript is omitted, we mean $\DC_\omega$. While $\DC_\lambda$ does imply $\AC_\lambda$, we do not have the reverse implication; see \cite[Theorem~8.1, Theorem~8.9]{jech_axiom_1973} for proofs.

\subsubsection{Comparability and dual comparability}
Finally, for a set $X$, we shall denote by $W_X$ the \emph{axiom of comparability (to $X$)}: for all $Y$, either $\abs{Y}\leq\abs{X}$ or $\abs{X}\leq\abs{Y}$. Similarly, we shall define the \emph{dual axiom of comparability (to $X$)}, $W_X^*$, to be the statement that for all $Y$, either $\abs{Y}\leq^*\abs{X}$ or $\abs{X}\leq^*\abs{Y}$. We immediately observe the following:
\begin{prop}\label{prop:svc-comparability-bound} $W_\lambda$ is equivalent to the statement ``for all $X$, either $X$ is well-orderable or $\aleph(X)\geq\lambda^+$''. Likewise, $W_\lambda^*$ is equivalent to the statement ``for all $X$, either $X$ is well-orderable or $\aleph^*(X)\geq\lambda^+$''.\qed
\end{prop}

\begin{defn}
Since $(\forall\lambda)\DC_\lambda$, $(\forall\lambda)W_\lambda$, and $(\forall\lambda)W_\lambda^*$ are all equivalent to $\AC$, whenever $M$ is a model of $\ZF+\lnot\AC$ we shall denote by $\lambda_\DC$ (respectively $\lambda_W$, $\lambda_W^*$) the least cardinal $\lambda$ such that $\DC_\lambda$ (respectively $W_\lambda$, $W_\lambda^*$) does not hold.
\end{defn}

\subsubsection{Small violations of choice}\label{s:preliminaries;ss:svc}

In \cite{blass_injectivity_1979}, the author introduces a choice-like axiom called \emph{small violations of choice}, also written $\SVC$. At its inception, it was defined by setting $\SVC(S)$ to be ``for all $X$ there is an ordinal $\eta$ and a surjection $f\colon\eta\times S\to X$'', where $S$ is a set (known as the \emph{seed}). We then use $\SVC$ to mean $(\exists S)\SVC(S)$. However, this is equivalent to several other statements.
\begin{fact}[{\cite{blass_injectivity_1979},\cite{usuba_choiceless_2021}}]\label{fact:svc-tfae}
The following are equivalent:
\begin{enumerate}[label=\textup{(\arabic*)}]
\item $M\vDash\SVC$;
\item\label{fact:svc-tfae;condition:injseed} $M\vDash``$There is a set $A$ such that for all $X$ there is an ordinal $\eta$ and an injection $f\colon X\to A\times\eta"$;
\item there is an inner model $V\subseteq M$ such that $V\vDash\ZFC$ and there is a symmetric system $\tup{\bbP,\sG,\sF}\in V$ such that $M=\HS_\sF^G$ for some $V$-generic $G\subseteq\bbP$;
\item there is an inner model $V\subseteq M$ such that $V\vDash\ZFC$ and there is $x\in M$ such that $M=V(x)$; and
\item there is a notion of forcing $\bbP\in M$ such that $\1_\bbP\forces\AC$.
\end{enumerate}
\end{fact}

\begin{defn}[Injective Seed]
We shall say that a set $A$ is an \emph{injective seed} for $M\vDash\ZF$ if it satisfies Condition~\ref*{fact:svc-tfae;condition:injseed}. That is, for all $X\in M$ there is an ordinal $\eta$ and an injection $f\colon X\to A\times\eta$ in $M$.
\end{defn}

\begin{prop}
$M\vDash\SVC$ if and only if $M$ has an injective seed.
\end{prop}

\begin{proof}
Suppose that $A$ is an injective seed for $M$. Then $M\vDash\SVC(A)$.

On the other hand, suppose that $S$ is a seed for $M$. We claim that $\power(S)$ is an injective seed for $M$. Indeed, suppose that $X\in M$, and let ${f\colon\eta\times S\to X}$ be a surjection. For $x\in X$, let $\alpha_x=\min\{\alpha<\eta\mid(\exists s\in S)f(\alpha,s)=x\}$. Since $f$ is a surjection, $x\mapsto\alpha_x$ is well-defined. Then define $g\colon X\to\power(S)\times\eta$ via ${g(x)=\tup{\{s\in S\mid f(\alpha_x,s)=x\},\alpha_s}}$. Then whenever $g(x)=g(y)$, we have ${\alpha_x=\alpha_y}$, and thus for all $s$ such that $f(\alpha_x,s)=x$, we have $f(\alpha_y,s)=y$. However, ${f(\alpha_x,s)=f(\alpha_y,s)}$, so $x=y$ as required.
\end{proof}

\subsection{Forcing}\label{s:preliminaries;ss:forcing}

By a \emph{notion of forcing} we mean a preordered set $\bbP$ with maximum element denoted $\1_\bbP$, or with the subscript omitted when clear from context. We write $q\leq p$ to mean that $q$ \emph{extends} $p$. Two conditions $p,p'$ are said to be \emph{compatible}, written $p\comp p'$, if they have a common extension. We follow Goldstern's alphabet convention so $p$ is never a stronger condition than $q$, etc.

When given a collection of $\bbP$-names, $\{\ddx_i\mid i\in I\}$, we will denote by $\{\ddx_i\mid i\in I\}^\bullet$ the canonical name this class generates: $\{\tup{\1,\ddx_i}\mid i\in I\}$. The notation extends naturally to ordered pairs and functions with domains in the ground model. Given a set $x$, the check name for $x$ is defined inductively as $\check{x}=\{\check{y}\mid y\in x\}^\bullet$.

\subsubsection{Symmetric extensions}

It is key to the role of forcing that if $V\vDash\ZFC$, and $G$ is $V$-generic for some notion of forcing $\bbP\in V$, then $V[G]\vDash\ZFC$. However, this demands additional techniques for trying to establish results that are inconsistent with $\AC$. Symmetric extensions extend the technique of forcing in this very way by constructing an intermediate model between $V$ and $V[G]$ that is a model of $\ZF$.

Given a notion of forcing $\bbP$, we shall denote by $\Aut(\bbP)$ the collection of automorphisms of $\bbP$. Let $\bbP$ be a notion of forcing and $\pi\in\Aut(\bbP)$. Then $\pi$ extends naturally to act on $\bbP$-names by recursion: $\pi\ddx=\{\tup{\pi p,\pi\ddy}\mid\tup{p,\ddy}\in\ddx\}$.

Such automorphisms extend to the forcing relation in the following way, proved in \cite[Lemma~14.37]{jech_set_2003}.
\begin{lem}[The Symmetry Lemma]
Let $\bbP$ be a notion of forcing, $\pi\in\Aut(\bbP)$, and $\ddx$ a $\bbP$-name. Then $p\forces\vphi(\ddx)$ if and only if $\pi p\forces\vphi(\pi\ddx)$.\qed
\end{lem}
\noindent Note in particular that for all $\pi\in\Aut(\bbP)$ we have $\pi\1=\1$. Therefore, $\pi\check{x}=\check{x}$ for all ground model sets $x$, and $\pi\{\ddx_i\mid i\in I\}^\bullet=\{\pi\ddx_i\mid i\in I\}^\bullet$, similarly extending to tuples, functions, etc.

Given a group $\sG$, a \emph{filter of subgroups} of $\sG$ is a set $\sF$ of subgroups of $\sG$ that is closed under supergroups and finite intersections. We say that $\sF$ is \emph{normal} if whenever $H\in\sF$ and $\pi\in\sG$, then $\pi H\pi^{-1}\in\sF$.

A \emph{symmetric system} is a triple $\PGF$ such that $\bbP$ is a notion of forcing, $\sG$ is a group of automorphisms of $\bbP$, and $\sF$ is a normal filter of subgroups of $\sG$. Given such a symmetric system, we say that a $\bbP$-name $\ddx$ is \emph{$\sF$-symmetric} if $\sym_\sG(\ddx)=\{\pi\in\sG\mid\pi\ddx=\ddx\}\in\sF$. $\ddx$ is \emph{hereditarily $\sF$-symmetric} if this notion holds for every $\bbP$-name hereditarily appearing in $\ddx$. We denote by $\HS_\sF$ the class of hereditarily $\sF$-symmetric names. When clear from context, we will omit subscripts and simply write $\sym(\ddx)$ or $\HS$. The following theorem, \cite[Lemma~15.51]{jech_set_2003}, is then key to the study of symmetric extensions.
\begin{thm}
Let $\PGF$ be a symmetric system, $G\subseteq\bbP$ a $V$-generic filter, and let $M$ denote the class $\HS^G_\sF=\{\ddx^G\mid\ddx\in\HS_\sF\}$. Then $M$ is a transitive model of $\ZF$ such that $V\subseteq M\subseteq V[G]$.\qed
\end{thm}
Finally, we have a forcing relation for symmetric extensions $\forces^\HS$ defined by relativising the forcing relation $\forces$ to the class $\HS$. This relation has the same properties and behaviour of the standard forcing relation $\forces$. Moreover, when $\pi\in\sG$, the Symmetry Lemma holds for $\forces^\HS$.

\section{$\AC_\WO$}\label{s:ac-wo}

$\AC_\WO$, the axiom of choice for all well-orderable families of non-empty sets, is known to be equivalent to the statement $(\forall X)\aleph(X)=\aleph^*(X)$, and the proof make use of the idea of transferring eccentricity upwards. This idea is best explained by proving the theorem.

\begin{thm}[{\cite{pelc_on_1978}}]\label{thm:acwo-equiv-pelc}
$\AC_\WO$ is equivalent to $(\forall X)\aleph(X)=\aleph^*(X)$.
\end{thm}

\begin{proof}
$(\implies)$. Let $X$ be a set. We always have that $\aleph(X)\leq\aleph^*(X)$, so it is sufficient to prove that $\aleph^*(X)\leq\aleph(X)$. Let $\lambda<\aleph^*(X)$, and let $f\colon X\to\lambda$ be a surjection. Since $f$ is a surjection, if we set $C=\{f^{-1}(\alpha)\mid\alpha<\lambda\}$ then $C$ is a well-ordered family of non-empty sets, and so by $\AC_\WO$, there is a choice function $c\colon\lambda\to X$. However, $c$ must be an injection since $f^{-1}(\alpha)\cap f^{-1}(\beta)=\emptyset$ whenever $\alpha\neq\beta$.

$(\impliedby)$. We shall prove $\AC_{\aleph_\delta}$ for all ordinals $\delta$ by induction. Suppose that we have established $\AC_{{<}\aleph_\delta}$ (indeed, this is a theorem of $\ZF$ for $\delta=0$), and let $X=\{X_\alpha\mid\alpha<\aleph_\delta\}$ where $X_\alpha\neq\emptyset$ for all $\alpha<\aleph_\delta$. By induction, for all $\alpha<\aleph_\delta$, $Y_\alpha\defeq\prod_{\gamma<\alpha}X_\gamma\neq\emptyset$. Define by induction on $\alpha<\aleph_\delta$ the cardinal $\kappa_\alpha$ and the set $D_\alpha$ in the following way:
\begin{equation*}
\kappa_\alpha\defeq\aleph\left(\bigcup\{D_\beta\mid\beta<\alpha\}\right)\text{ and }D_\alpha\defeq Y_\alpha\times\kappa_\alpha.
\end{equation*}
Let $D=\bigcup\{D_\alpha\mid\alpha<\aleph_\delta\}$ and $\lambda=\sup\{\kappa_\alpha\mid\alpha<\aleph_\delta\}$. By projection to its second co-ordinate, there is a surjection $D\to\lambda$, and so $\aleph^*(D)\geq\lambda^+$. By assumption, we must also have that $\aleph(D)\geq\lambda^+$. Let $f\colon\lambda\to D$ be an injection.

Since $\lambda>\aleph(\bigcup_{\beta<\alpha}D_\beta)$ for all $\alpha<\delta$, it cannot be the case that $f``\lambda\subseteq\bigcup_{\beta<\alpha}D_\beta$ for any $\alpha<\delta$. Therefore, by projection to its first co-ordinate, $f``\lambda$ gives a well-ordered set of partial choice functions for $X$ of unbounded domain. Setting ${c(\alpha)=\first(f(\gamma))(\alpha)}$, where \(\first\) is the projection $\tup{a,b}\mapsto a$ and $\gamma$ is minimal such that $\alpha\in\dom(\first(f(\gamma)))$, produces a choice function $c\in\prod X$ as desired.
\end{proof}

Inspired by this proof, we produce a general framework for taking a set $X$ and producing a set $D$ of larger Lindenbaum number with some control over the Hartogs and Lindenbaum numbers produced.

\begin{defn}
Let $\kappa$ be a cardinal, $\delta>0$ a limit ordinal, and $X=\{X_\alpha\mid\alpha<\delta\}$ be such that for all $\alpha<\delta$, $Y_\alpha\defeq\prod_{\beta<\alpha}X_\beta\neq\emptyset$. Inductively define the cardinals $\kappa_\alpha$ and sets $D_\alpha$ for $\alpha<\delta$ as follows:
\begin{equation*}
\kappa_\alpha\defeq\aleph\parenth{\midcup\{D_\beta\mid\beta<\alpha\}}+\kappa\text{ and }D_\alpha\defeq Y_\alpha\times\kappa_\alpha.
\end{equation*}
We then define the \emph{upwards transfer construction} $D=D(X,\kappa)$ as $\bigcup_{\alpha<\delta}D_\alpha$ and ${\lambda=\lambda(X,\kappa)}$ as ${\sup\{\kappa_\alpha\mid\alpha<\delta\}}$. We observe that $\lambda>\kappa$ and that $\lambda$ is a limit cardinal.
\end{defn}

\begin{prop}\label{prop:the-machine}
$\aleph(D)\geq\lambda$, $\aleph^*(D)\geq\lambda^+$, and if $\aleph(D)\geq\lambda^+$ then $\prod_{\alpha<\delta}X_\alpha$ is non-empty.
\end{prop}

\begin{proof}
Let $\mu<\lambda$. Then there is $\alpha<\delta$ such that $\mu<\kappa_\alpha$. Hence by fixing $y\in Y_\alpha$, the function $\gamma\mapsto\tup{y,\gamma}$ is an injection $\mu\to D$. Therefore, $\aleph(D)\geq\lambda$. On the other hand, by projection to the second co-ordinate we have that $\aleph^*(D)\geq\lambda^+$.

Finally, suppose that $\aleph(D)\geq\lambda^+$, so there is an injection $f\colon\lambda\to D$. Note that since $\lambda>\aleph(\bigcup_{\beta<\alpha}D_\beta)$ for all $\alpha<\delta$, we cannot have that $f``\lambda\subseteq\bigcup_{\beta<\alpha}D_\beta$ for any $\alpha<\delta$. Hence, $f``\lambda$ intersects $D_\alpha$ for unboundedly many $\alpha$s and so, by projection to the first co-ordinate and the well-order of $f``\lambda$, we may select some $y_\alpha\in Y_\alpha$ for unboundedly many $\alpha$s. Putting these partial choice functions together yields $c\in\prod X$ as desired. Explicitly, $c(\alpha)=\first(f(\gamma))(\alpha)$, where $\gamma<\delta$ is minimal such that $\alpha\in\dom(\first(f(\gamma)))$.
\end{proof}

With upwards transfer construction in hand, we may produce a great many new statements that are all equivalent to $\AC_\WO$ through the general framework of Theorem~\ref{thm:acwo-equiv-pelc}.

\begin{thm}\label{thm:ac-wo-hartlin}
The following are equivalent:
\begin{enumerate}[label=\textup{(\arabic*)}]
\item\label{thm:ac-wo-hartlin;condition:hartislin} For all $X$, $\aleph(X)=\aleph^*(X)$;
\item\label{thm:ac-wo-hartlin;condition:hartislin-evLin} there is $\kappa$ such that for all $X$, $\aleph^*(X)\geq\kappa\implies\aleph(X)=\aleph^*(X)$;
\item\label{thm:ac-wo-hartlin;condition:hartislin-evHart} there is $\kappa$ such that for all $X$, $\aleph(X)\geq\kappa\implies\aleph(X)=\aleph^*(X)$;
\item\label{thm:ac-wo-hartlin;condition:acwo} $\AC_\WO$;
\item\label{thm:ac-wo-hartlin;condition:suc} for all $X$, $\aleph(X)$ is a successor; and
\item\label{thm:ac-wo-hartlin;condition:reg} for all $X$, $\aleph(X)$ is regular.
\end{enumerate}
\end{thm}

\begin{proof}
(${\ref*{thm:ac-wo-hartlin;condition:hartislin}\implies\ref*{thm:ac-wo-hartlin;condition:hartislin-evLin}}$) Immediate.

(${\ref*{thm:ac-wo-hartlin;condition:hartislin-evLin}\implies\ref*{thm:ac-wo-hartlin;condition:hartislin-evHart}}$) Immediate from $\aleph^*(X)\geq\aleph(X)$ for all $X$.

(${\ref*{thm:ac-wo-hartlin;condition:acwo}\implies\ref*{thm:ac-wo-hartlin;condition:hartislin}}$) See the proof of Theorem~\ref{thm:acwo-equiv-pelc}.

(${\ref*{thm:ac-wo-hartlin;condition:acwo}\implies\ref*{thm:ac-wo-hartlin;condition:suc}}$) Let $X$ be such that $\aleph(X)\geq\lambda$ for a limit cardinal $\lambda$. For all ${\mu<\lambda}$, recall that $\Inj{\mu}{X}$ denotes the set of injections $\mu\to X$, and use $\AC_\WO$ to find ${F\in\prod_{\mu<\lambda}\Inj{\mu}{X}}$. Define the sequence $\tup{x_\alpha\mid\alpha<\eta}$ by concatenating the sequences $\tup{F(\mu)(\beta)\mid\beta<\mu}$ and removing duplicate entries. This must be an injection $\eta\to X$ for some ordinal $\eta$, and we claim that $\eta\geq\lambda$, proving that $\aleph(X)\geq\lambda^+$ as required. If instead $\eta<\lambda$, then letting $\eta<\mu<\lambda$ we find that $F(\mu)``\mu\subseteq\{x_\alpha\mid\alpha<\eta\}$, so $F(\mu)$ is not an injection, a contradiction. Hence $\eta\geq\lambda$ as required.

(${\ref*{thm:ac-wo-hartlin;condition:acwo}\implies\ref*{thm:ac-wo-hartlin;condition:reg}}$) By $\AC_\WO$, all successor cardinals are regular.\footnote{If $\alpha<\kappa^+$ and $f\colon\alpha\to\kappa^+$ is a strictly increasing sequence, then by $\AC_\WO$ we may simultaneously pick injections $f(\alpha)\to\kappa$ for all $\alpha$, and so $\abs{\midcup f``\alpha}\leq\kappa\times\kappa=\kappa$.} Also by Condition~\ref*{thm:ac-wo-hartlin;condition:acwo}, we deduce Condition~\ref*{thm:ac-wo-hartlin;condition:suc}, so $\aleph(X)$ is a successor for all $X$. Therefore, $\aleph(X)$ is regular for all $X$.

(Each of (3), (5), and (6) implies (4)) Since the proofs of these three implications are almost identical, they have been packaged here.

We shall prove $\AC_{\aleph_\delta}$ by induction on $\delta$. Suppose that we have $\AC_{{<}\aleph_\delta}$ (which is a theorem of $\ZF$ in the case of $\delta=0$), and let $X=\{X_\alpha\mid\alpha<\aleph_\delta\}$ where $X_\alpha\neq\emptyset$ for all $\alpha<\aleph_\delta$. By induction, $\prod_{\beta<\alpha}X_\beta\neq\emptyset$ for all $\alpha<\aleph_\delta$, so setting $D=D(X,\kappa+\aleph_\delta)$ and $\lambda=\lambda(X,\kappa+\aleph_\delta)$ we get that $\aleph(D)\geq\lambda>\kappa+\aleph_\delta$. For each of the Conditions~\ref*{thm:ac-wo-hartlin;condition:hartislin-evHart}, \ref*{thm:ac-wo-hartlin;condition:suc}, and \ref*{thm:ac-wo-hartlin;condition:reg} we may use Proposition~\ref{prop:the-machine} to prove that $\prod X\neq\emptyset$ by showing that $\aleph(D)\geq\lambda^+$.

First assume Condition~\ref*{thm:ac-wo-hartlin;condition:hartislin-evHart}. Then $\aleph(D)\geq\kappa$, so $\aleph(D)=\aleph^*(D)\geq\lambda^+$, and thus $\prod X\neq\emptyset$.

If we instead assume Condition~\ref*{thm:ac-wo-hartlin;condition:suc}, then since $\lambda$ is a limit cardinal we have ${\aleph(D)\geq\lambda^+}$, and thus $\prod X\neq\emptyset$.

Finally, if we assume Condition~\ref*{thm:ac-wo-hartlin;condition:reg}, then since $\lambda=\sup\{\kappa_\alpha\mid\alpha<\aleph_\delta\}$ and $\lambda>\aleph_\delta$, we get that $\lambda$ is a singular cardinal, so $\aleph(D)\geq\lambda^+$, and thus $\prod X\neq\emptyset$.

Hence, in each case, we can conclude Condition~\ref*{thm:ac-wo-hartlin;condition:acwo}.
\end{proof}

\begin{rk}
Note that, by combining the techniques exhibited in the proof of Theorem~\ref{thm:ac-wo-hartlin}, one can produce a vast collection of conditions that are equivalent to $\AC_\WO$. For example, $\AC_\WO$ is equivalent to the statement ``there is $\kappa$ such that for all $X$, if $\aleph^*(X)\geq\kappa$ then either $\aleph(X)$ is a successor or $\aleph(X)$ is regular''.
\end{rk}

\section{Spectra}\label{s:spectra}

Theorem~\ref{thm:ac-wo-hartlin} gives the very strong conclusion that $\AC_\WO$ is not just equivalent to $(\forall X)\aleph(X)=\aleph^*(X)$, but also that $\aleph(X)$ is a successor cardinal for all $X$. In this way, $\AC_\WO$ gives us that the class of all pairs $\tup{\aleph(X),\aleph^*(X)}$ is minimal, that is just $\{\tup{\lambda^+,\lambda^+}\mid\lambda\in\Card\}$. However, it is possible to violate $\AC_\WO$, and we inspect here the various ways in which this can be violated in models of $\SVC$.

\begin{defn}[Hartogs--Lindenbaum Spectrum]
For a model $M$ of $\ZF$, the \emph{(Hartogs--Lindenbaum) spectrum} of $M$, denoted by $\Spec_\aleph(M)$, is the class
\begin{equation*}
\Spec_\aleph(M)\defeq\{\tup{\lambda,\kappa}\mid M\vDash(\exists X)\aleph(X)=\lambda,\aleph^*(X)=\kappa\}
\end{equation*}
of all possible pairs $\tup{\lambda,\kappa}$ such that there is $X\in M$ with $\aleph(X)=\lambda$ and $\aleph^*(X)=\kappa$.
\end{defn}

In \cite{karagila_hartogs_2023_draft}, the authors show that it is consistent with $\ZF$ to have a model $M$ such that
\begin{equation*}
\Spec_\aleph(M)=\{\tup{\lambda,\kappa}\mid\aleph_0\leq\lambda\leq\kappa\}\cup\{\tup{n+1,n+1}\mid n<\omega\}.
\end{equation*}
However, this was achieved with a class-length iteration of symmetric extensions and (by the \ref{thm:main-thm}) cannot be optimised further. Recall that $M\vDash\SVC$ if $M$ is a symmetric extension of a ground model $V\vDash\AC$. We shall show that in this case $\Spec_\aleph(M)$ is controlled on a tail of cardinals. In particular, there is $\Omega\in\Card$ such that for all $\lambda\geq\Omega$, if $\aleph(X)=\lambda$ then $\aleph^*(X)\leq\lambda^+$.

\subsection{The spectrum of Cohen's first model}\label{s:spetra;ss:cohens}

Let us first establish the methods that will be employed in our favourite test model of $\ZF+\SVC+\lnot\AC$: Cohen's first model. In fact, Cohen's first model is not even a model of $\AC_\omega$, and so in particular is not a model of $\AC_\WO$. This will be a consequence of the existence of a set $A$ such that $\aleph(A)=\aleph_0$; by \cite[Section~2.4.1]{jech_axiom_1973}, if $\AC_{\aleph_0}$ holds then every infinite set has a countably infinite subset.

We begin with a quick reminder of how Cohen's first model is defined. Let $V$ be a model of $\ZFC$ and, working in $V$, let $\bbP=\Add(\omega,\omega)$, that is the forcing whose conditions are finite partial functions $p\colon\omega\times\omega\to2$ with the ordering $q\leq p$ if $q\supseteq p$. The group $\sG$ is the finitary permutations of $\omega$, those $\pi\in S_\omega$ such that $\{n<\omega\mid\pi n\neq n\}$ is finite. This has group action on $\bbP$ given by $\pi p(\pi n,m)=p(n,m)$. For $E\in[\omega]^\lomega$, let $\fix(E)\leq\sG$ be the subgroup $\{\pi\in\sG\mid\pi\res E=\id\}$, and let $\sF$ be the filter of subgroups of $\sG$ generated by the $\fix(E)$ as $E$ varies over $[\omega]^\lomega$. Let $M$ be the symmetric extension of $V$ by this symmetric system, and let $V[G]$ be the full extension by $\bbP$. Note that, since $\bbP$ is c.c.c., $V$, $M$, and $V[G]$ will agree on the cardinalities and cofinalities of ordinals, and $V$ and $V[G]$ will agree on the cardinalities of all sets in $V$.

For each $n<\omega$, let $\dda_n$ be the $\bbP$-name $\{\tup{p,\check{m}}\mid p(n,m)=1\}$ and note that for all $\pi\in\sG$, $\pi\dda_n=\dda_{\pi n}$. Let $\ddA=\{\dda_n\mid n<\omega\}^\bullet$, so $\pi\ddA=\ddA$ for all $\pi\in\sG$. Let $A$ be the realisation of the name $\ddA$ in $M$. Note that $V[G]\vDash\abs{A}=\aleph_0$ witnessed by, say, $\{\tup{\check{n},\dda_n}^\bullet\mid n<\omega\}^\bullet$.

\begin{fact}[{\cite[$\S5.5$]{jech_axiom_1973}}]\label{fact:cohen-model}
The following hold in $M$:
\begin{enumerate}[label=\textup{(\arabic*)}]
\item $\aleph(A)=\aleph_0$;
\item for every infinite $X$ there is a surjection $f\colon X\to\omega$; and
\item for every $X$ there is an ordinal $\eta$ and an injection $f\colon X\to[A]^\lomega\times\eta$.
\end{enumerate}
\end{fact}

An immediate corollary of this fact is the following.

\begin{cor}\label{cor:cohen;linden-A-is-aleph1}
In $M$, $\aleph^*(A)=\aleph_1$.
\end{cor}

\begin{proof}
By Statement (2) in Fact~\ref{fact:cohen-model}, there is a surjection $A\to\omega$ in $M$, and so $\aleph^*(A)\geq\aleph_1$. However, in the outer model, $\abs{A}=\aleph_0$, and so certainly there is no surjection $A\to\aleph_1$ in $M$. Hence, $\aleph^*(A)=\aleph_1$.
\end{proof}

We shall use these facts alongside the techniques laid out in Section~\ref{s:ac-wo} to produce a complete picture of the spectrum of Cohen's first model.

\begin{thm}\label{thm:cohen-spectrum}
$\Spec_\aleph(M)=\{\tup{\lambda^+,\lambda^+}\mid\lambda\in\Card\}\cup\{\tup{\lambda,\lambda^+}\mid\cf(\lambda)=\aleph_0\}$.
\end{thm}

To translate this to the notation of the \ref{thm:main-thm}, we have that $\phi=\psi=\aleph_0$, $\chi_0=\chi=\Omega=\psi^*=\aleph_1$, and $C=\{\aleph_0\}$. Hence,
\begin{equation*}
\Spec_\aleph(M)=\bigcup\left\{\begin{alignedat}{2}
\Suc&=\{\tup{\lambda^+,\lambda^+}&&\mid\lambda\in\Card\}\\
\frakD&\subseteq\{\tup{\lambda,\kappa}&&\mid\aleph_0\leq\lambda\leq\kappa\leq\aleph_1,\aleph_1\leq\kappa\}\\
\frakC&\subseteq\{\tup{\lambda,\lambda^+}&&\mid\cf(\lambda)=\aleph_0,\lambda<\aleph_1\}\\
\frakU&=\{\tup{\lambda,\lambda^+}&&\mid\cf(\lambda)=\aleph_0,\lambda\geq\aleph_1\}.
\end{alignedat}\right.
\end{equation*}
Therefore, both $\frakD$ and $\frakC$ are subsets of $\{\tup{\aleph_0,\aleph_1}\}$. However, we have already seen by Fact~\ref{fact:cohen-model} and Corollary~\ref{cor:cohen;linden-A-is-aleph1} that $\aleph(A)=\aleph_0$ and $\aleph^*(A)=\aleph_1$, and thus $\frakD=\frakC=\{\tup{\aleph_0,\aleph_1}\}$.

\subsubsection{The upper bound}

We begin with an upper bound for the spectrum.

\begin{lem}\label{lem:cohen-hartlin-cf}
For all $X$, if $\aleph(X)=\lambda$ and $\aleph^*(X)=\kappa$, then for all $\mu\in[\lambda,\kappa)$, $\cf(\mu)\leq\aleph_0$.
\end{lem}

\begin{proof}
We may assume that $X\subseteq[A]^\lomega\times\eta$ for some ordinal $\eta$. For each $a\in[A]^\lomega$, let $X_a=X\cap(\{a\}\times\eta)$, so each $X_a$ is well-orderable. Let $\mu\in[\lambda,\kappa)$. Since $\mu<\kappa$, there is a surjection $f\colon X\to\mu$. Hence, in $V[G]$, we have that $\mu=\bigcup_{a\in[A]^\lomega}f``X_a$, so $\mu$ is the union of countably many sets (note that $A$ is countable in $V[G]$, and so $[A]^\lomega$ is as well). If $\mu$ has uncountable cofinality then there is $a\in[A]^\lomega$ such that $V[G]\vDash\abs{X_a}\geq\mu$. Since $M$ and $V[G]$ agree on the cardinalities of sets of ordinals, and $X_a$ is well-orderable, we have that $\abs{X_a}\geq\mu$ in $M$ as well. However, this contradicts $\mu\geq\aleph(X)$. Hence we must have $\cf(\mu)=\aleph_0$.
\end{proof}

\begin{lem}
For all $X$, $\aleph^*(X)$ is a successor.
\end{lem}

\begin{proof}
As before, we may assume that $X\subseteq[A]^\lomega\times\eta$ for some ordinal $\eta$, and again we shall denote by $X_a$ the set $X\cap(\{a\}\times\eta)$ for all $a\in[A]^\lomega$. Let $\kappa$ be a limit cardinal such that $\aleph^*(X)\geq\kappa$. Then we must show that $\aleph^*(X)>\kappa$, that is we must show that there is a surjection $X\to\kappa$. Let $\mu<\kappa$ be infinite. Since $\kappa$ is a limit, we still have that $\mu^+<\kappa$, and $\mu^+$ is regular in both $M$ and $V[G]$. Since $\mu^+<\kappa$, there is a surjection $f\colon X\to\mu^+$, so in $V[G]$ we have $\mu^+=\bigcup\{f``X_a\mid a\in[A]^\lomega\}$ a countable union. Since $\cf(\mu^+)>\aleph_0$, there is $a\in[A]^\lomega$ such that $\abs{f``X_a}\geq\mu^+$ in $V[G]$, and so $\abs{X_a}\geq\mu^+$ in $V[G]$. However, $X_a$ is well-orderable, and so $\abs{X_a}\geq\mu^+$ in $M$ as well. Therefore, for all $\mu<\kappa$, there is $a\in[A]^\lomega$ such that $\abs{X_a}>\mu$. Hence the projection of $X$ onto its second co-ordinate is a surjection onto a subset of $\eta$ of cardinality at least $\kappa$, and so this can be turned into a surjection $X\to\kappa$.
\end{proof}

\begin{cor}\label{cor:no-reverse-implication}
Both ``for all $X$, $\aleph^*(X)$ is regular'' and ``for all $X$, $\aleph^*(X)$ is a successor'' are strictly weaker than $\AC_\WO$.\qed
\end{cor}

\begin{cor}\label{cor:cohen}
If $\aleph(X)=\lambda$ and $\aleph^*(X)=\kappa$, then one of the following holds:
\begin{enumerate}[label=\textup{(\alph*)}]
\item $\kappa=\lambda$ are successors; or
\item $\kappa=\lambda^+$ and $\cf(\lambda)=\aleph_0$.
\end{enumerate}
\end{cor}

\begin{proof}
Suppose that $\kappa=\lambda$. Since $\kappa$ is a successor, $\lambda$ must be as well. Suppose instead that $\kappa>\lambda$. Then by Lemma~\ref{lem:cohen-hartlin-cf}, for all $\mu\in[\lambda,\kappa)$, $\cf(\mu)=\aleph_0$. However, $\cf(\lambda^+)=\lambda^+>\aleph_0$ in $M$, so we have $\lambda^+\notin[\lambda,\kappa)$, i.e. $\kappa=\lambda^+$.
\end{proof}

\subsubsection{The lower bound}

To complete Theorem~\ref{thm:cohen-spectrum}, we must show the lower bound for the spectrum.

\begin{lem}\label{lem:cohen-lower-bound}
For all $\lambda$ such that $\cf(\lambda)=\aleph_0$ there is a set $X$ such that $\aleph(X)=\lambda$ and $\aleph^*(X)=\lambda^+$.
\end{lem}

\begin{proof}
By Corollary~\ref{cor:cohen}, it will be sufficient to prove that there is a set $X$ such that $\aleph(X)=\lambda$. By Fact~\ref{fact:cohen-model}, suppose that $\lambda>\aleph_0$. Let $\tup{\lambda_n\mid n<\omega}$ be a strictly increasing sequence of cardinals such that $\sup\{\lambda_n\mid n<\omega\}=\lambda$. For each $n<\omega$, let $X_n=\lambda_n\times\Inj{n}{A}$, and let $X=\bigcup_{n<\omega}X_n$. We certainly have that $\aleph(X)\geq\lambda$, as for all $n<\omega$ there is an injection $f_n\colon\lambda_n\to X$ by taking arbitrary $c\in\Inj{n}{A}$, and having $f_n(\alpha)=\tup{\alpha,c}$. Suppose now that there were an injection $f\colon\lambda\to X$. Note that in $V[G]$, $\abs{\Inj{n}{A}}=\aleph_0$, and so $\abs{\bigcup_{m<n}\lambda_m\times\Inj{m}{A}}<\lambda$ for all $n<\omega$. Therefore, $f``\lambda$ is not a subset of $\bigcup_{m<n}\lambda_m\times\Inj{m}{A}$ for any $n<\omega$, so $f$ must give us a well-ordered collection of injections $n\to A$ for arbitrarily large $n$. We may now put these injections together to produce an injection $\omega\to A$, contradicting $\aleph(A)=\aleph_0$. Hence an injection $\lambda\to X$ cannot exist and $\aleph(X)=\lambda$ as desired.
\end{proof}

Combining Corollary~\ref{cor:cohen} and Lemma~\ref{lem:cohen-lower-bound}, we immediately obtain Theorem~\ref{thm:cohen-spectrum}:
\begin{equation*}
\Spec_\aleph(M)=\{\tup{\lambda^+,\lambda^+}\mid\lambda\in\Card\}\cup\{\tup{\lambda,\lambda^+}\mid\cf(\lambda)=\aleph_0\}.
\end{equation*}

\subsection{The spectrum of a model of $\SVC$}\label{s:spectra;ss:svc}

When dealing with Cohen's first model, having an outer model of $\ZFC$ that agrees on the cardinalities and cofinalities of ordinals was an important fact that appeared in almost every proof of the previous section. Fortunately, this is not unique to Cohen's first model, and is very close to the conclusions that can be drawn from $\SVC$. Let $M$ be a model of $\SVC+\lnot\AC$, witnessed by an inner model $V$ of $\ZFC$, a symmetric system $\tup{\bbP,\sG,\sF}$, and an injective seed $A$. Let $V[G]$ be the outer forcing extension of $M$, so $V\subseteq M\subseteq V[G]$ for $V$-generic $G\subseteq\bbP$, and $M=\HS_\sF^G$. For a set $X\in M$, denote by $\cabs{X}$ the least ordinal $\delta$ such that $V[G]\vDash``\abs{X}=\abs{\delta}$ and $\delta$ is a cardinal$"$. Unlike in Cohen's model, we may not have that $\cabs{\kappa}=\kappa$ for all cardinals $\kappa\in M$ since $\bbP$ may collapse some cardinals, but by appealing to large enough cardinals we are able to overcome this obstacle.

\begin{fact}[{\cite[Theorem~15.3]{jech_set_2003}}]\label{fact:ccc-card-cof-agree}
There is a cardinal $\lambda$ (e.g. $\abs{\bbP}^+$) such that, if ${\alpha,\beta\geq\lambda}$ are ordinals, then $V\vDash\abs{\alpha}=\abs{\beta}$ if and only if $V[G]\vDash\abs{\alpha}=\abs{\beta}$. Furthermore, if $\eta\geq\lambda$ is a cardinal then $V\vDash\cf(\alpha)=\eta$ if and only if $V[G]\vDash\cf(\alpha)=\eta$.
\end{fact}

The following is an immediate corollary of Fact~\ref{fact:ccc-card-cof-agree}.

\begin{cor}
There is a cardinal $\lambda$ such that $V$, $M$, and $V[G]$ agree on cardinalities above $\lambda$ and cofinalities greater than $\lambda$.\qed
\end{cor}

Let $\tau$ be the least cardinal such that $M$ and $V[G]$ agree on cardinalities at least $\tau$ and cofinalities greater than or equal to $\tau$. Finally, fix an injective seed $A$ for $M$ and let $\nu=\cabs{A}$.

Throughout this section, all sets and statements about sets are understood to be in the context of $M$ unless stated otherwise.

\subsubsection{An upper bound}

We first aim to create an upper bound on the Hartogs--Lindenbaum spectrum of $M$ by showing scenarios in which combinations of Hartogs and Lindenbaum numbers are not possible.

The proof of the following is effectively identical to the proof of Lemma~\ref{lem:cohen-hartlin-cf}.

\begin{lem}\label{lem:svc-hartlin-cf}
Suppose that $\aleph(X)=\lambda$ and $\aleph^*(X)=\kappa$. Then for all $\mu\in[\lambda,\kappa)$, $\cf(\mu)<\max(\nu^+,\tau)$.\qed
\end{lem}

\begin{cor}\label{cor:svc-oblate-cardinals}
If $\aleph(X)=\lambda\geq\max(\nu^+,\tau)$ then $\aleph^*(X)\leq\lambda^+$.
\end{cor}

\begin{proof}
Since $\lambda\geq\tau$, and $V[G]\vDash\cf(\lambda^+)=\lambda^+$, we have ${M\vDash\cf(\lambda^+)=\lambda^+}$, and in particular $M\vDash\cf(\lambda^+)>\max(\nu^+,\tau)$. Therefore, by Lemma~\ref{lem:svc-hartlin-cf}, $\lambda^+\notin[\lambda,\aleph^*(X))$, so we must have that $\aleph^*(X)\leq\lambda^+$.
\end{proof}

\begin{lem}\label{lem:svc-hartlin-suc}
For all $X$, if $\aleph^*(X)>\max(\tau,\nu)$ then $\aleph^*(X)$ is a successor cardinal.
\end{lem}

\begin{proof}
Let $\kappa>\max(\tau,\nu)$ be a limit cardinal and suppose that $X$ is such that $\aleph^*(X)\geq\kappa$. Then we must show that $\aleph^*(X)\geq\kappa^+$.

We may assume that $X\subseteq A\times\eta$ for some $\eta\in\Ord$. For each $a\in A$, let $X_a=X\cap(\{a\}\times\eta)$, so each $X_a$ is well-orderable. We aim to show that, if $\pi\colon X\to\eta$ is the projection of $X$ to its second co-ordinate, then $\abs{\pi``X}\geq\kappa$ and thus $\aleph^*(X)>\kappa$ as required.

Let $\mu\in(\max(\tau,\nu),\kappa)$. Since $\kappa>\max(\tau,\nu)$ is a limit cardinal, such a $\mu$ exists, and indeed $\sup\{\mu\in\Card\mid\max(\tau,\nu)<\mu<\kappa\}=\kappa$. Since $\mu<\kappa$ and $\kappa$ is a limit cardinal, $\mu^+<\kappa$, so there is a surjection $f\colon X\to\mu^+$. In $V[G]$, $\cf(\mu^+)=\mu^+>\nu$, so there is $a\in A$ such that $\abs{f``X_a}\geq\mu^+$. Since $\mu$ was taken arbitrarily, for all $\mu<\kappa$ there is $a\in A$ such that $\abs{X_a}\geq\mu$, and thus $\abs{\pi``X}\geq\kappa$.
\end{proof}

\begin{lem}\label{lem:svc-dc-bound}
Suppose that $M\vDash\DC_\mu$, and let $\lambda$ be a limit cardinal such that $\cf(\lambda)=\mu$. Then for all $X\in M$, $\aleph(X)\neq\lambda$.
\end{lem}

\begin{proof}
Suppose that $\aleph(X)\geq\lambda$, and let $\tup{\lambda_\alpha\mid\alpha<\mu}$ be a continuous strictly increasing cofinal sequence in $\lambda$ with $\lambda_0=0$. Define the tree $T$ to be the set $\bigcup_{\alpha<\mu}\Inj{\lambda_\alpha}{X}$, with $f\leq g$ if $f\subseteq g$. Since $\aleph(X)\geq\lambda$, $T$ has no maximal nodes and, by taking unions, we see that $T$ is $\mu$-closed. By $\DC_\mu$, $T$ has a maximal branch $b$. Let $f=\bigcup b$. Then $f$ is an injection $\kappa\to X$ for some $\kappa$, and $\kappa$ is either $\lambda_\alpha$ for some $\alpha<\mu$ or is $\lambda$. In the former case, we may extend $f$ to a function from $\lambda_{\alpha+1}$ by inspecting any node at level $\alpha+1$ ($\aleph(X)>\lambda_{\alpha+1}$ so it is nonempty) and removing any duplicate entries. We'll remove fewer than $\lambda_{\alpha+1}$ entries, so adjoining the two will still yield an extension in $T$. This contradicts the maximality of $b$. Therefore, we are in the latter case and $f$ is an injection $\lambda\to X$, so $\aleph(X)>\lambda$.
\end{proof}

Recall that $\lambda_\DC$ is the least cardinal $\lambda$ such that $\DC_\lambda$ does not hold.

\begin{prop}\label{prop:svc-upper-bound}
If $\aleph(X)=\lambda$ and $\aleph^*(X)=\kappa$, then one of the following holds:
\begin{enumerate}[label=\textup{(\alph*)}]
\item\label{prop:svc-upper-bound;case:bounded} $\kappa\leq\max(\nu^+,\tau^+)$;
\item\label{prop:svc-upper-bound;case:suc} $\kappa=\lambda$ are successors; or
\item\label{prop:svc-upper-bound;case:cf} $\kappa=\lambda^+$ and $\lambda_{\DC}\leq\cf(\lambda)<\max(\nu^+,\tau)$.
\end{enumerate}
\end{prop}

\begin{proof}
If $\kappa\leq\max(\nu^+,\tau^+)$ then we are in Case~\ref*{prop:svc-upper-bound;case:bounded}, so we may assume otherwise. By Lemma~\ref{lem:svc-hartlin-suc}, $\kappa$ is a successor cardinal. If $\lambda=\kappa$, then we are in Case~\ref*{prop:svc-upper-bound;case:suc}, so assume that $\lambda<\kappa$. Suppose for contradiction that $\kappa>\lambda^+$. Then setting $\mu=\max(\lambda^+,\nu^+,\tau^+)$, we have $\mu\in[\lambda,\kappa)$, so by Lemma~\ref{lem:svc-hartlin-cf} $\cf(\mu)<\max(\nu^+,\tau)$. However, $V[G]\vDash\cf(\mu)=\mu\geq\tau$, since $\mu$ is a successor cardinal, and so we have $M\vDash\cf(\mu)=\mu\geq\max(\nu^+,\tau)$, contradicting that $\cf(\mu)<\max(\nu^+,\tau)$. Therefore, $\kappa=\lambda^+$. Finally, we have that $\lambda_{\DC}\leq\cf(\lambda)<\max(\nu^+,\tau)$ by Lemma~\ref{lem:svc-hartlin-cf} and Lemma~\ref{lem:svc-dc-bound}.
\end{proof}

The underlying pattern of Proposition~\ref{prop:svc-upper-bound} is that once we have dealt with the chaos of $\aleph^*(X)\leq\max(\nu^+,\tau^+)$ and the inevitability of $\aleph(X)=\aleph^*(X)=\lambda^+$, all that we have are sets $X$ with $\aleph(X)=\lambda$ and $\aleph^*(X)=\lambda^+$ for some cardinal $\lambda$. Indeed, this scenario is the only one in which we may have an eccentric set of arbitrarily large Hartogs or Lindenbaum number.

\begin{defn}[Oblate cardinal]
An \emph{oblate cardinal} is a cardinal $\lambda$ such that there is a set $X$ with $\aleph(X)=\lambda$ and $\aleph^*(X)=\lambda^+$.
\end{defn}

As we have observed, the only candidates for oblate cardinals that are at least as large as $\max(\nu^+,\tau)$ are those singular cardinals $\lambda$ with $\cf(\lambda)\in[\lambda_{\DC},\max(\nu^+,\tau))$. However, this does not tell us which of those cardinals will be oblate. Fortunately, when we begin to produce a lower bound for the spectrum, we will have very strong results for transferring eccentricity that force cardinals to be oblate.

Recall that $\lambda_W$ is the least cardinal $\lambda$ such that $W_\lambda$ does not hold, and that $\lambda_W^*$ is defined analogously for $W_\lambda^*$. Combining Proposition~\ref{prop:svc-comparability-bound} and Proposition~\ref{prop:svc-upper-bound}, we produce an upper bound of $\Spec_\aleph(M)$ in three parts:

\begin{enumerate}
\item The successors, $\Suc=\{\tup{\lambda^+,\lambda^+}\mid\lambda\in\Card\}$. In fact, we always have that $\Suc\subseteq\Spec_\aleph(M)$ since $\aleph(\lambda)=\aleph^*(\lambda)=\lambda^+$ for all cardinals $\lambda$.
\item The `bounded chaos', those $\tup{\lambda,\kappa}$ where $\lambda_W\leq\lambda\leq\kappa\leq\max(\nu^+,\tau^+)$ and $\lambda_W^*\leq\kappa$. If we have $\aleph(X)^+<\aleph^*(X)$ for some set $X$, then it must appear here.
\item The oblate cardinals, those $\tup{\lambda,\lambda^+}$ with $\cf(\lambda)\in[\lambda_{\DC},\max(\nu^+,\tau))$.
\end{enumerate}

\subsubsection{A lower bound}

Having just seen the upper bound of the spectrum of a model of $\SVC$, we now wish to exhibit a lower bound, which we shall construct entirely through controlling oblate cardinals. By Theorem~\ref{thm:ac-wo-hartlin}, if any eccentric set exists then we must have eccentric sets of arbitrarily large Hartogs or Lindenbaum number. However, by Proposition~\ref{prop:svc-upper-bound}, if $\aleph(X)\geq\max(\nu^+,\tau)$ and $X$ is eccentric, then we must have that $\aleph(X)$ is an oblate cardinal. Therefore, there must be a proper class of oblate cardinals. In this section we will explore methods of lifting oblate cardinals to larger oblate cardinals with the same cofinality. We do this in two ways: The first lifts any set $B$ with $\aleph(B)$ a limit cardinal, and the other lifts any set $B$ with $\aleph(B)<\aleph^*(B)$.

\begin{lem}\label{lem:svc-hartlin-construction}
Let $B$ be such that $\aleph(B)=\mu$ a limit cardinal, $\lambda$ a cardinal with $\lambda\geq\nu+\tau+\cabs{B^{{<}\mu}}^+$, and $\tup{\delta_\alpha\mid\alpha<\mu}$ a strictly increasing sequence of ordinals such that ${\sup\{\delta_\alpha\mid\alpha<\mu\}=\lambda}$.\footnote{Such a sequence exists if and only if $\mu\leq\lambda$ and $\cf(\lambda)=\cf(\mu)$.} Then there is a set $X$ such that $\aleph(X)=\lambda$ and $\aleph^*(X)=\lambda^+$.
\end{lem}

\begin{proof}
Let $\tup{\delta_\alpha\mid\alpha<\mu}$ be the sequence described in the proposition. For each $\beta<\mu$, let $X_\beta=\bigcup_{\alpha<\delta}\delta_\alpha\times\Inj{\alpha}{B}$, and let $X=\bigcup\{X_\beta\mid\beta<\mu\}$. Note that since $\aleph(B)=\mu$, $\Inj{\alpha}{B}\neq\emptyset$ for all $\alpha<\mu$. Note also that for all $\beta<\mu$, in $V[G]$ we can embed $X_\beta$ into $\delta_\beta\times B^{{<}\mu}$, so $\cabs{X_\beta}\leq\cabs{\delta_\beta}\cdot\cabs{B^{{<}\mu}}<\lambda$. Therefore, there cannot be an injection $\lambda\to X_\beta$ for any $\beta<\mu$.

For all $\alpha<\mu$, $\delta_\alpha$ embeds into $X$, so $\aleph(X)\geq\lambda$. Furthermore, projection to its first co-ordinate begets a surjection $X\to\lambda$, so $\aleph^\ast(X)\geq\lambda^+$. Suppose that there were an injection $\lambda\to X$. We have already seen that the image of this injection cannot be contained in any $X_\beta$, so it must be unbounded. By our usual construction, noting that $\mu$ is a limit, this gives us an injection $\mu\to B$, contradicting $\aleph(B)=\mu$. Therefore, $\aleph(X)=\lambda$ as required. By Proposition~\ref{prop:svc-upper-bound}, $\aleph^\ast(X)=\lambda^+$.
\end{proof}

\begin{lem}\label{lem:svc-hartlin-upscale}
Let $B$ be such that $\aleph(B)=\mu$ and $\aleph^*(B)>\mu$, $\lambda$ a cardinal with $\lambda\geq\nu+\tau+\cabs{B}^+$, and $\tup{\lambda_\alpha\mid\alpha<\mu}$ a strictly increasing sequence such that ${\sup\{\lambda_\alpha\mid\alpha<\mu\}=\lambda}$. Then there is a set $X$ such that $\aleph(X)=\lambda$ and $\aleph^*(X)=\lambda^+$.
\end{lem}

\begin{proof}
Let $\tup{\lambda_\alpha\mid\alpha<\mu}$ be the sequence described in the proposition and $f\colon B\to\mu$ a surjection. For each $\alpha<\mu$, let $Y_\alpha=f^{-1}(\alpha)\neq\emptyset$, and let $X\defeq\bigcup_{\alpha<\mu}\lambda_\alpha\times Y_\alpha$. By the usual techniques, noting that $\lambda>\cabs{B}$, we see that $\aleph(X)\geq\lambda$ but any injection $\lambda\to X$ would beget an injection $\mu\to B$, contradicting $\aleph(B)=\mu$. Furthermore, $\aleph^\ast(X)\geq\lambda^+$ and by Proposition~\ref{prop:svc-upper-bound}, $\aleph^\ast(X)=\lambda^+$.
\end{proof}

\begin{prop}\label{prop:svc-upwards-c}
There is a set $C\subseteq[\lambda_{\DC},\max(\nu^+,\tau))$ and a cardinal $\Omega$ such that for all $\lambda\geq\Omega$, $\lambda$ is an oblate cardinal if and only if $\cf(\lambda)\in C$.
\end{prop}

\begin{proof}
Firstly, by Lemma~\ref{lem:svc-hartlin-cf} and Lemma~\ref{lem:svc-dc-bound}, if $\lambda$ is an oblate cardinal then $\cf(\lambda)\in[\lambda_{\DC},\max(\nu^+,\tau))$. For each $\mu\in[\lambda_{\DC},\max(\nu^+,\tau))$, if there is an oblate cardinal $\lambda$ with $\cf(\lambda)=\mu$, then let $\Omega_\mu$ be the minimal value of $\nu+\tau+\cabs{B}^+$ among all sets $B$ such that $\cf(\aleph(B))=\mu$ and $\aleph^*(B)=\aleph(B)^+$. By Lemma~\ref{lem:svc-hartlin-upscale}, for each $\lambda>\Omega_\mu$ with $\cf(\lambda)=\mu$, we can construct a set $X$ witnessing that $\lambda$ is an oblate cardinal. Let $C=\{\mu\in[\lambda_{\DC},\max(\nu^+,\tau))\mid(\exists\lambda\in\Card)\cf(\lambda)=\mu$ and $\lambda$ is oblate$\}$. Then, setting $\Omega=\sup\{\Omega_\mu\mid\mu\in C\}$, we see that for all $\lambda\geq\Omega$, we have that $\lambda$ is oblate if and only if $\cf(\lambda)\in C$.
\end{proof}

In fact, we get more than Proposition~\ref{prop:svc-upwards-c}. By Lemma~\ref{lem:svc-hartlin-upscale}, $C$ is precisely the set $\{\cf(\aleph(X))\mid\aleph(X)<\aleph^*(X)\}$.

Finally, putting together Proposition~\ref{prop:svc-comparability-bound}, Proposition~\ref{prop:svc-upper-bound}, and Proposition~\ref{prop:svc-upwards-c}, we obtain the following theorem.

\begin{mainthm}\makeatletter\def\@currentlabel{main theorem}\makeatother\label{thm:main-thm}
Let $M\vDash\SVC$. Then there are cardinals $\phi\leq\psi\leq\chi_0\leq\Omega$, a cardinal $\psi^*\geq\psi$, a cardinal $\chi\in[\chi_0,\chi_0^+]$, and a set $C\subseteq[\phi,\chi_0)$ such that
\begin{equation*}
\Spec_\aleph(M)=\bigcup\left\{\begin{alignedat}{2}
\Suc&=\{\tup{\lambda^+,\lambda^+}&&\mid\lambda\in\Card\}\\
\frakD&\subseteq\{\tup{\lambda,\kappa}&&\mid\psi\leq\lambda\leq\kappa\leq\chi,\psi^*\leq\kappa\}\\
\frakC&\subseteq\{\tup{\lambda,\lambda^+}&&\mid\cf(\lambda)\in C,\lambda<\Omega\}\\
\frakU&=\{\tup{\lambda,\lambda^+}&&\mid\cf(\lambda)\in C,\lambda\geq\Omega\}.
\end{alignedat}\right.
\end{equation*}
\end{mainthm}

In the notation of this section, $\phi=\lambda_{\DC}$, $\psi=\lambda_W$, $\psi^*=\lambda_W^*$, $\chi_0=\max(\nu^+,\tau)$, and $\chi=\max(\nu^+,\tau^+)$.

\begin{rk}
The techniques used in Section~\ref{s:spetra;ss:cohens} to produce cleaner, stricter bounds only rely on the total agreement of cardinalities and cofinalities of ordinals between Cohen's model and the outer model. Any other model of $\SVC$ in which this occurs will have similarly tight control over the spectrum, assuming that one can find an injective seed (which is no small feat).
\end{rk}

\begin{eg}
Suppose that $M\vDash\SVC+\lnot\AC_\WO$ and $M$ agrees with its outer model on the cardinalities and cofinalities of all ordinals. Let $A$ be an injective seed for $M$ and suppose that we have preserved $\DC_{{<}\cabs{A}}$, so $\lambda_{\DC}=\cabs{A}$. Then, in the notation of the \ref*{thm:main-thm}, $\phi=\cabs{A}$, $\chi_0=\chi=\cabs{A}^+$, and so $C=\{\cabs{A}\}$. Since $\phi\leq\psi\leq\chi_0$, we have that $\psi=\cabs{A}$ or $\cabs{A}^+$. If $\psi=\cabs{A}^+$, then $\frakD$ is empty (other than maybe $\tup{\cabs{A}^+,\cabs{A}^+}$, which is included in $\Suc$), and if $\psi=\cabs{A}$ then $\frakD\subseteq\{\tup{\cabs{A},\cabs{A}},\tup{\cabs{A},\cabs{A}^+}\}$. However, given that $\cabs{A}\in C$, we may as well exclude $\tup{\cabs{A},\cabs{A}^+}$ from $\frakD$, as it would appear in $\frakC$ anyway.

Finally, we have the oblate cardinals. Let $\Omega$ be the least value of $\cabs{B}$ where $\cf(\aleph(B))=\cabs{B}$, and $\aleph^*(B)>\aleph(B)$ with the restriction that $\Omega\geq\chi_0=\cabs{A}^+$. Then by Lemma~\ref{lem:svc-hartlin-upscale} (using that $\tau=\aleph_0$), we have that $\Omega\leq\aleph(X)<\aleph^*(X)$ if and only if $\cf(\aleph(X))=\cabs{A}$. Putting this together,
\begin{equation*}
\Spec_\aleph(M)=\bigcup\left\{\begin{aligned}
\Suc&=\{\tup{\lambda^+,\lambda^+}\mid\lambda\in\Card\}\\
\frakD&\subseteq\{\tup{\cabs{A},\cabs{A}}\}\\
\frakC&\subseteq\{\tup{\lambda,\lambda^+}\mid\cf(\lambda)=\cabs{A},\lambda<\Omega\}\\
\frakU&=\{\tup{\lambda,\lambda^+}\mid\cf(\lambda)=\cabs{A},\lambda\geq\Omega\}.
\end{aligned}\right.
\end{equation*}
By $\DC_{{<}\cabs{A}}$, $\aleph(\cabs{A})\geq\cabs{A}$, and certainly $\aleph(A)\leq\cabs{A}^+$. However, if $\aleph(A)=\cabs{A}$ and $\aleph^*(A)=\cabs{A}^+$, then we have $\Omega=\cabs{A}^+$, and $\frakC=\{\tup{\cabs{A},\cabs{A}^+}\}$. On the other hand, if $\aleph(A)=\aleph^*(A)=\cabs{A}$ is a limit cardinal, then we have $\Omega\leq\cabs{A^{{<}\cabs{A}}}^+$, so there may be oblate cardinals missing.

A more specific example is Cohen's model which, as we saw, follows this pattern precisely with $\Omega=\cabs{A}^+=\aleph_0$ and $\frakD=\emptyset$. In Question~\ref{qn:d-not-empty} we ask if it is possible to have $\frakD\neq\emptyset$ in this situation.
\end{eg}

\section{The future}

Chief among unanswered questions in this field is as follows:
\begin{qn}
Precisely which spectra are possible to achieve with models of $\SVC$?
\end{qn}
Throughout the work in Section~\ref{s:spectra;ss:svc}, the spectre of the injective seed haunted our calculations. If we wished to decide precisely which spectra are achievable in models of $\SVC$, then more information on the injective seeds of the models that we produce is required. It is also unclear to me how much control we can have over $C$, $\frakC$ and $\Omega$. In \cite{karagila_hartogs_2023_draft}, for each infinite $\lambda\leq\kappa$, the authors construct a symmetric system $\tup{\bbP,\sG,\sF}$ such that for all $V\vDash\ZFC$, $\1_\bbP\forces^{\HS}(\exists X)\aleph(X)=\check{\lambda},\aleph^*(X)=\check{\kappa}$. In fact, such models can be constructed so that the model and the outer model agree with the cardinality and cofinality of all ordinals, just as in the case of Cohen's model.

\begin{qn}
What are the spectra of the symmetric extensions produced in \cite{karagila_hartogs_2023_draft}?
\end{qn}

Assume $\SVC$ and suppose that $\aleph(X)=\lambda$ and $\aleph^*(X)=\kappa$, with $\lambda^+<\kappa$, and suppose that $\lambda_{\DC}=\lambda$. Then for all $\mu\in[\lambda,\kappa)$, we have that $\aleph(X+\mu)=\mu^+$ and $\aleph^*(X+\mu)=\kappa$, so $\cf(\mu^+)\in C$ for all $\mu\in[\lambda,\kappa)$ and, as we have seen, $\cf(\lambda)\in C$ as well. Therefore, given any $\mu\in[\lambda,\kappa)$, if $\cf(\mu)$ is a successor, then we can deduce that $\cf(\mu)\in C$ if and only if $\cf(\mu)\geq\lambda$. However, this does not complete the picture. If there is a weakly inaccessible cardinal\footnote{An uncountable regular limit cardinal.} $\mu\in[\lambda,\kappa)$, then we have not been able to deduce if $\mu\in C$ using these tools.

\begin{qn}
Let $\mu$ be weakly inaccessible and suppose that for some set $X$, $\aleph(X)<\mu<\aleph^*(X)$. Must there exist $Y$ such that $\aleph(Y)=\mu$?
\end{qn}

\begin{qn}
Suppose that $M\vDash\ZF+\SVC+\AC_\WO$. Can we always force $\AC$ in a way that collapses no cardinals?
\end{qn}

\begin{qn}\label{qn:d-not-empty}
Is there a model $M$ of $\SVC+\lnot\AC_\WO$ such that
\begin{equation*}
\Spec_\aleph(M)=\Suc\cup\{\tup{\lambda,\lambda^+}\mid\cf(\lambda)=\aleph_0\}\cup\{\tup{\aleph_0,\aleph_0}\},
\end{equation*}
as described in the example after the \ref*{thm:main-thm}? If we replace $\aleph_0$ by an arbitrary infinite cardinal $\kappa$, for which $\kappa$ is this possible?
\end{qn}

The Bristol model, introduced in \cite{karagila_bristol_2018} and expanded upon in \cite{karagila_approaching_2020}, is an inner model of $L[c]$, where $c$ is a single Cohen real. It therefore satisfies the same conditions that made manipulating Cohen's model so nice: There is an outer model of $\ZFC$ that agrees with the inner model on all cardinalities and cofinalities of ordinals. However, the Bristol model was constructed with explicit intention to violate $\SVC$, and so many of the techniques used in this paper cannot be applied `as is'. However, it does not seem too far-fetched that an amount of this work can be reclaimed. The Bristol model is, in a sense, a limit of models of $\SVC$ that approaches the final model, and it is reasonable to believe that one can look at intermediate models to obtain results about the spectrum.

\begin{qn}[{\cite[Question~10.19]{karagila_approaching_2020}}]
Does the Bristol model satisfy $\AC_\WO$? If not, what is the spectrum of the Bristol model?
\end{qn}

\section{Acknowledgements}

The author would like to thank the reviewer for their feedback on this article. The author would also like to thank Asaf Karagila for his feedback on an early version of the paper and for encouraging the author to write out their results to share with the mathematical community. It can be difficult to realise when one's work is in a state that is ready to be shared, and without that push the author may have sat on this research until their viva.

\providecommand{\bysame}{\leavevmode\hbox to3em{\hrulefill}\thinspace}
\providecommand{\MR}{\relax\ifhmode\unskip\space\fi MR }
\providecommand{\MRhref}[2]{%
  \href{http://www.ams.org/mathscinet-getitem?mr=#1}{#2}
}
\providecommand{\href}[2]{#2}

\end{document}